\documentclass[english,12pt]{article}
\usepackage[utf8]{inputenc}
\usepackage{amssymb,amsmath,amsfonts,amsthm,rotating,setspace,graphicx,float}
\usepackage{amssymb,amsmath,amsthm,rotating,setspace}
\usepackage{enumerate}
\usepackage{tikz}
\usepackage{esint}
\usepackage{mathrsfs}
\usepackage{geometry,url}
\usepackage{bbm}
\usepackage{paralist}
\usepackage[hidelinks]{hyperref}
\usepackage[title]{appendix}
\usepackage{latexsym}
\usepackage{epsfig}
\usepackage{color}

\newtheorem{thm}{Theorem}[section]
\newtheorem{lemma}[thm]{Lemma}
\newtheorem{prop}[thm]{Proposition}
\newtheorem{cor}[thm]{Corollary}

\theoremstyle{definition}
\theoremstyle{definition}

\numberwithin{equation}{section}

\def\SS{S^{N-1}}
\def\o{\Omega}
\def\oo{\overline\Omega}

\def\ve{\varepsilon}
\def\K1{K_{\alpha, \beta}}

\def\R{\mathbb{R}}

\def\Co{\mathcal{C}_{\Omega}}

\newcommand{\abs}[1]{\lvert#1\rvert}

\title{Elliptic inequalities with nonlinear convolution and Hardy terms in cone-like domains}
\author{Marius Ghergu\footnote{School of Mathematics and Statistics,
        University College Dublin, Belfield, Dublin 4, Ireland;
        {\tt marius.ghergu@ucd.ie}; ORCID: 0000-0001-9104-5295} \;\footnote{Institute of Mathematics Simion Stoilow of the Romanian Academy, 21 Calea Grivitei St., 010702 Bucharest, Romania}
        $\;\;$            and    $\;$
    Zhe Yu\footnote{School of Mathematics and Statistics,
        University College Dublin, Belfield, Dublin 4, Ireland;
        {\tt zhe.yu@ucdconnect.ie }; ORCID: 0000-0002-2261-4981} }
\date{}
\begin{document}

\maketitle

\begin{abstract}
\vspace{0.3cm} 
\noindent We study the inequality $\displaystyle -\Delta u -  \frac{\mu}{\abs{x}^2} u \geq (|x|^{-\alpha} * u^p)u^q$ in an unbounded cone $\mathcal{C}_\Omega^\rho\subset \mathbb{R}^N$ ($N\geq 2$) generated by a subdomain $\Omega$ of the unit sphere $S^{N-1}\subset \mathbb{R}^N,$ $p, q, \rho>0$,  $\mu\in \mathbb{R}$ and  $0\leq \alpha < N$.  In the above, $|x|^{-\alpha} * u^p$ denotes the standard convolution operator in the cone $\mathcal{C}_\Omega^\rho$. We discuss the existence and nonexistence of positive solutions in terms of $N, p, q, \alpha, \mu$ and $\Omega$.  Extensions to systems of inequalities are also investigated.
\end{abstract}

\noindent{\bf Keywords:} Semilinear elliptic inequalities; Hardy terms; nonlinear convolution term; cone-like domains; a priori estimates

\medskip

\noindent{\bf 2020 AMS MSC:} 35J15; 35J47; 35A23; 35A01

\section{Introduction}
Let $\o$ be a subdomain of the unit sphere $\SS \subset \R^N (N \geq 2)$ such that $\o$ is connected and relatively open in $\SS$. Set
\[\mathcal{C}_{\Omega} = \{(r, \omega) \in \R^N: r > 0 \text{ and } \omega \in \Omega\}.\]
For any $\rho > 0$ and $0 \leq \rho_1 < \rho_2 < \infty$, we define
\begin{equation*}
\begin{split}
\mathcal{C}_{\Omega}^{\rho} &= \{(r, \omega) \in \R^N: r > \rho \text{ and } \omega \in \Omega\},\\
\mathcal{C}_{\Omega}^{\rho_1, \rho_2} &= \{(r, \omega) \in \R^N: \rho_1 < r < \rho_2 \text{ and } \omega \in \Omega\}.
\end{split}
\end{equation*}
This paper is concerned with the study of positive solutions to the following semi-linear elliptic inequality with convolution and Hardy terms:
\begin{equation}\label{eq0}
{\mathscr L}_H u :=  -\Delta u -  \frac{\mu}{\abs{x}^2} u \geq (|x|^{-\alpha}  * u^p)u^q \;\;\;\text{ in }\;\; \mathcal{C}_{\Omega}^{\rho}.
\end{equation}
In the above inequality we assume $\mu\in \R$, $p, q>0$  and 
$0\leq \alpha<N$. The quantity $|x|^{-\alpha}\ast u^p$ represents the convolution operation in $\mathcal{C}_\Omega^\rho$ given by
$$ 
|x|^{-\alpha}\ast
u^p=\int_{\mathcal{C}_\Omega^\rho} \frac{u^p(y)}{|x-y|^{\alpha}} dy\,, \quad x\in \mathcal{C}_\Omega^\rho.
$$ 

\medskip

In this paper we shall be interested in classical positive solutions of \eqref{eq0}, that is, functions $u \in  C^2(\mathcal{C}_{\Omega}^{\rho})\cap C(\overline{\mathcal{C}_{\Omega}^{\rho}})$ such that  
\begin{equation}\label{KKK}
u(x)>0\;, \;\; |x|^{-\alpha} * u^p<\infty\quad\mbox{  for all }\; x\in {\mathcal C}_\Omega^\rho,
\end{equation}
and $u$ satisfies \eqref{eq0} pointwise in $\mathcal{C}_{\Omega}^\rho$. We do not prescribe any condition on  $u$ on the boundary $\partial \mathcal{C}_{\Omega}^\rho$ of the cone.

The prototype model
\begin{equation}\label{hartree}
-\Delta u = \big(|x|^{-1}\ast u^2\big) u \quad\mbox{ in }\; \R^3, 
\end{equation}
has been around for nearly a century, being introduced in 1928 by D.R. Hartree \cite{Har28a,Har28b,Har28c} in the form of time-dependent equations 
given by
\begin{equation}\label{hartree0}
i\psi_t+\Delta \psi+(|x|^{-1} \ast |\psi|^2)\psi=\psi\quad\mbox{in }{\mathbb R}^3\times (0, \infty).
\end{equation}
We may regard \eqref{hartree0} in connection with the Schr\"odinger-Newton (or Schr\"odinger-Poisson) equation, namely
\begin{equation}\label{spoisson}
i\psi_t+\Delta \psi+\phi \psi=\psi\quad\mbox{in }{\mathbb R}^3\times (0, \infty),  
\end{equation}
where $\phi$ is a gravitational potential representing the interaction of the particle with its own gravitational field and  satisfies the Poisson equation
$$
\Delta \phi=4\pi G|\psi|^2\quad\mbox{ in } \R^3.
$$ 
Thus, one may write the gravitational potential  as a convolution, namely 
$$
\phi(x)=4\pi G \big(|x|^{-1}\ast |\psi|^2\big),
$$ 
which plugged into \eqref{spoisson} leads us (modulo a scaling of coefficients) to \eqref{hartree0}. Looking for a solution $\psi$  of \eqref{hartree0} in the form $\psi(x,t)=e^{-it}u(x)$ we see that $u$ satisfies \eqref{hartree}. 

More recently, the equation \eqref{hartree} is encountered in the literature under the name of Choquard or Choquard-Pekar equation. Precisely,  \eqref{hartree} was introduced in 1954  by S.I. Pekar \cite{Pek54} as a model in quantum theory of a Polaron at rest (see also
\cite{DA10}) and in 1976, P. Choquard  used \eqref{hartree} in a certain approximation to Hartree-Fock theory of one component plasma (see \cite{Lie76}). 
The first mathematical study of \eqref{hartree} appeared in the mid 1970s and is due to E.H. Lieb \cite{Lie76}, followed by P.-L. Lions' works \cite{Lio80,Lio84}.

The inequality
\begin{equation}\label{mv}
-\Delta u \geq (|x|^{-\alpha}\ast u^p)u^q\quad\mbox{ in }\R^N\setminus\overline{B}_1,
\end{equation}
was first discussed in  \cite{MV13a}. Let us note that \eqref{mv} corresponds to \eqref{eq0} in the cone $\mathcal{C}^1_\Omega$ where $\Omega=S^{N-1}$. The approach in \cite{MV13a} relies on the so-called Agmon-Allegretto-Piepenbrink positivity principle.
Quasilinear versions of \eqref{mv} (again in exterior domains)  are studied in \cite{GKS20} while the polyharmonic case is discussed in \cite{GMM21}( see also \cite{CZ1, CZ2, CZ3,  G23, GT16, GTbook} for more contexts in which \eqref{mv} occurs). 

To the best of our knowledge, this is the first work dealing with semilinear elliptic inequalities in cone-like domains featuring convolution terms. The local cases 
$$-\Delta u=u^p\quad\mbox{ and } \quad -\Delta u\geq u^p
$$ 
in a cone-like domain have been studied in \cite{B92, BE90, KLM05, KLMS05, LLM06}. In the current work we derive an a priori estimate  result which we provide in Proposition \ref{pconea1} below.  Our method is inspired by \cite{BP01} (see also \cite{GKS20}) where integral estimates are obtained to study local inequalities of type $-\Delta_p u\geq |x|^{\sigma}u^q$ in $\R^N\setminus B_1$. Due to the geometric features of our cone domain $\mathcal{C}_\Omega^\rho$, we cannot handle $W^{1,1}_{loc}$-solutions as in \cite{BP01} and \cite{GKS20}, and thus have to restrict ourselves to the class $C^2(\mathcal{C}_{\Omega}^{\rho})\cap C(\overline{\mathcal{C}_{\Omega}^{\rho}})$. The test function we construct (see \eqref{tf} below) is not a proper $C^\infty_c(\mathcal{C}_\Omega^\rho)$ function, rather a mere element in the space $C^2(\overline{\mathcal{C}_{\Omega'}^{R, 4R}})$, where $\Omega'\subset\subset \Omega$ is a subdomain and $R>\rho$ is large. Nonetheless, Proposition \ref{pconea1} below provides us with the right tool to derive the  nonexistence of a positive solution to \eqref{eq0} under some conditions on the exponents and parameters.

Apart from the a priori estimates described above, another crucial tool in our approach is the Hardy inequality on cone-like domains established in \cite{LLM06}:
\begin{equation}\label{hardy1}
\int_{\mathcal{C}_\Omega^\rho}|\nabla \varphi|^2\geq C_{H, \Omega}\int_{\mathcal{C}_\Omega^\rho}\frac{\varphi^2}{|x|^2} dx \quad\mbox{ for all }\varphi\in C^\infty_c(\mathcal{C}_\Omega^\rho),
\end{equation}
where $C_{H, \Omega}=\lambda_1 +(\frac{N - 2}{2})^2$ and $\lambda_1$ denotes the first eigenvalue of the Laplace-Beltrami operator on $\Omega$ subject to Dirichlet boundary condition. As obtained in \cite{LLM06} the constant $C_{H, \Omega}$ in \eqref{hardy1} is optimal.
Throughout this paper, we denote by $\phi$ the corresponding eigenfunction which we normalize as $\phi>0$ in $\Omega$ and $\max_{\overline\Omega}\phi=1$.

For $\mu\leq C_{H, \Omega}$ the quadratic equation
\begin{equation}\label{eqbeta}
\gamma(\gamma+N-2)=\lambda_1-\mu
\end{equation}
has two real roots which satisfy $\gamma_*<0\leq \gamma^*$. 
 
Our first result on \eqref{eq0} establishes conditions under which no positive solutions exist.
\begin{thm}\label{thm1} {\rm (Nonexistence)}

Assume $0\leq \alpha<N$, $p, q > 0$ and $\mu \in \R$.  
\begin{enumerate}
\item[{\rm (a)}]  If $\mu>C_{H, \Omega}$ then \eqref{eq0} has no positive solutions. 

\item[{\rm (b)}]  If $\mu\leq C_{H, \Omega}$ and one of the following conditions hold, then \eqref{eq0} has no positive solutions. 
\begin{enumerate}
\item[{\rm (i)}]  $\displaystyle  p\leq \frac{N-\alpha}{|\gamma_*|}$.
\smallskip

\item[{\rm (ii)}]  $\displaystyle p + q < 1 + \frac{N - \alpha + 2}{\abs{\gamma_*}}$.\smallskip

\item[{\rm (iii)}]   $\displaystyle  p + q = 1 + \frac{N - \alpha + 2}{|\gamma_*|}$ and $q> 1$.\smallskip

\end{enumerate}

\end{enumerate}

\end{thm}
\noindent We next turn to the existence  of positive solutions to \eqref{eq0}. Under the assumption  
\begin{equation}\label{qq}
q>1+\max\left\{0,\frac{2-\alpha}{|\gamma_*|}\right\},
\end{equation}
we provide optimal conditions in terms of $p, q, \alpha$ and $\mu$ for the existence of a positive solution. 
Precisely, we have:

\begin{thm}\label{thm2} {\rm (Existence)}

Assume $\mu\leq C_{H, \Omega}$, $0\leq \alpha<N$ and \eqref{qq} holds.
Then, inequality \eqref{eq0} has a positive solution in $\mathcal{C}_{\Omega}^{\rho}$ if and only if 
\begin{equation}\label{pq}
p>\frac{N-\alpha}{|\gamma_*|}\quad\mbox{ and }\;  \quad p + q > 1 + \frac{N - \alpha + 2}{|\gamma_*|}.
\end{equation}
\end{thm}
The existence and nonexistence regions for \eqref{eq0} in the positive quadrant of the $pq$-plane are depicted in Figure \ref{fig1} below in the case $0\leq \alpha<N$. 

\begin{figure}[ht!]\label{fig1}
\begin{center}
  \includegraphics[width=5.1in]{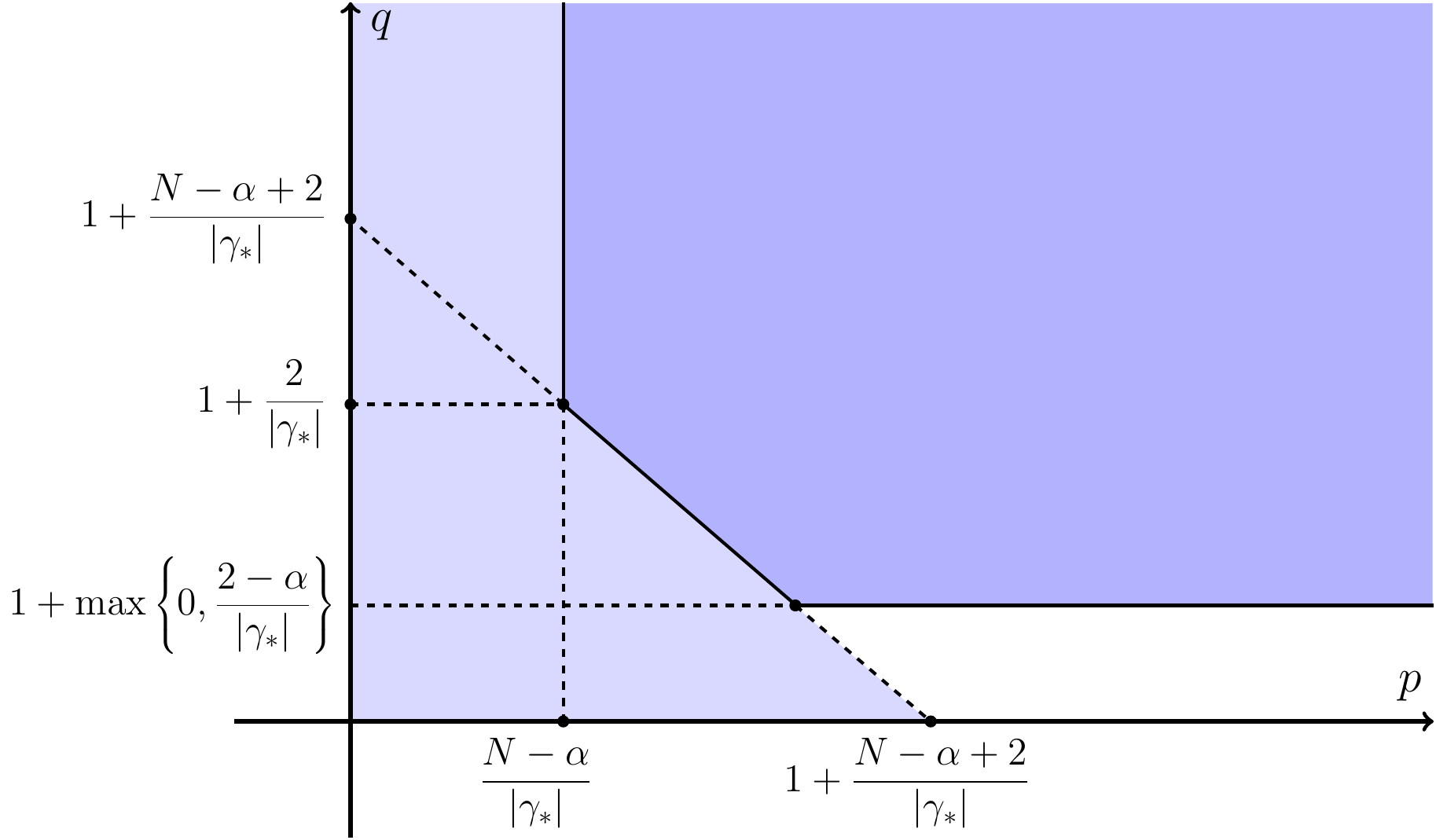}
  \caption{The existence region (dark shaded) and nonexistence region (light shaded) for positive solutions of \eqref{eq0} in the case $0\leq \alpha<N$. }
    \end{center}
\end{figure}
For $\alpha=0$ the inequality \eqref{eq0} reads 
\begin{equation}\label{zer}
\mathscr{L}_H u\geq \Big(\int_{\mathcal{C}_\Omega^\rho} u^p(y) dy\Big) u^q\quad\mbox{ in }\; \mathcal{C}_\Omega^\rho.
\end{equation}
As a consequence of Theorem \ref{thm1} and  Theorem \ref{thm2} we deduce the following result.

\begin{cor}\label{ao}
\begin{enumerate}
\item[{\rm (i)}] If $\mu> C_{H, \Omega}$ then  inequality \eqref{zer} has no positive solutions. 

\item[{\rm (ii)}]  If $\mu\leq C_{H, \Omega}$ and $q>1+\frac{2}{|\gamma_*|}$ then \eqref{zer} has a positive solution if and only if $p>\frac{N}{|\gamma_*|}$. 
\end{enumerate}
\end{cor}

\noindent Our approach can be extended to  system of inequalities of type
\begin{equation}\label{mainsys}
\begin{cases}
\displaystyle {\mathscr L}_{H} u :=  -\Delta u -  \frac{\mu}{\abs{x}^2} u \geq \big(|x|^{-\alpha}\ast v^p\big) v^q \\[0.1in]
\displaystyle {\mathscr L}_{H}  v :=  -\Delta v -  \frac{\mu}{\abs{x}^2} v \geq u^s
\end{cases}
\quad\mbox{ in }\; \mathcal{C}_\Omega^\rho,
\end{equation}
where $p,q>0$ and $s> 1$.  The system \eqref{mainsys} contains a combination of the local nonlinearity  $u^s$  and nonlocal nonlinearities given by the convolution operator. As in the case of a single inequality, we say that $(u,v)$ is a positive solution of \eqref{mainsys} if:

\begin{itemize}
\item $u,v\in   C^2(\mathcal{C}_{\Omega}^{\rho})\cap C(\overline{\mathcal{C}_{\Omega}^{\rho}})$ and $\big(|x|^{-\alpha}\ast v^p\big)v^q<\infty$ for all $ x\in \mathcal{C}_\Omega^\rho$.
\item $u,v>0$ in $\mathcal{C}_\Omega^\rho$ satisfy \eqref{mainsys} pointwise.
\end{itemize}
Semilinear elliptic systems of inequalities in cone-like domains 
were discussed in \cite{Lap02, MP01} (see also \cite{Lap02b}  for time dependent  inequalities). 
In the case of \eqref{mainsys} we are able to identify a critical curve that separates the existence and nonexistence regions of positive solutions to \eqref{mainsys}. 

The main result on the system \eqref{mainsys} is stated below.
\begin{thm}\label{thsysinq}
Assume $\mu\in \R$, $p,q>0$ and $0\leq \alpha<N$. 
\begin{enumerate}
\item[\rm (a)] If $\mu>C_{H, \Omega}$ then \eqref{mainsys} has no positive solutions.

\item[\rm (b)] Assume $\mu\leq C_{H, \Omega}$.
\begin{enumerate}
\item[\rm (i)] If $s> 1$, $p+q\geq 2$ and 
$$
p\leq \frac{N-\alpha}{|\gamma_{*}|}\quad\mbox{or}\quad p+q<\frac{1}{s}\Big(1+\frac{2}{|\gamma_{*}|}\Big)+\frac{N-\alpha+2}{|\gamma_{*}|},
$$
then \eqref{mainsys} has no positive solutions.
\item[\rm (ii)]  If $\displaystyle s > 1+\frac{2}{|\gamma_{*}|}$, $\displaystyle q>1+\max\Big\{0,\frac{2-\alpha}{|\gamma_{*}|}\Big\}$ and 
$$
p> \frac{N-\alpha}{|\gamma_{*}|}\quad\mbox{and}\quad  p+q>\frac{1}{s}\Big(1+\frac{2}{|\gamma_{*}|}\Big)+\frac{N-\alpha+2}{|\gamma_{*}|},
$$
then \eqref{mainsys} has  positive solutions.
\end{enumerate}
\end{enumerate}
\end{thm}

Theorem \ref{thsysinq} yields a new critical curve given by 
$$
p+q=\frac{1}{s}\Big(1+\frac{2}{|\gamma_{*}|}\Big)+\frac{N-\alpha+2}{|\gamma_{*}|},
$$ 
which defines the regions of existence and nonexistence of a positive solution to \eqref{mainsys}. We may regard the quantity $p+q$ as the global exponent of $v$ in the first inequality of \eqref{mainsys}.

\medskip

The remaining of this work is organised as follows. In Section 2 we collect some preliminary results consisting of various integral estimates. In Section 3 we derive an a priori estimate for positive solutions of \eqref{eq0}. This is one of the main tools in the proof of Theorem \ref{thm1} which is presented in Section 4. Section 5 is devoted to the proof of Theorem \ref{thm2} while the proof of Theorem \ref{thsysinq} is given in Section 6. 
Throughout this article by $c_0$, $c$, $C$, $C_1$, $C_2$, $\ldots$ we denote positive  constants whose value may change on every occurrence.

\section{Preliminary Results}
In this section we derive some preliminary results for our approach. 
\begin{lemma}\label{hrd} If $u\in C^2(\mathcal{C}_\Omega^\rho)$ is positive and satisfies $\mathscr{L}_H u\geq 0$ in $\mathcal{C}_\Omega^\rho$, then $\mu \leq  C_{H, \Omega}$.
\end{lemma}
\begin{proof} Let $\varphi \in C_c^{\infty}(\mathcal{C}_{\Omega}^\rho)$. Multiply the inequality $\mathscr{L}_H u\geq 0$ by $\frac{\varphi^2}{u}$ and integrate over $\mathcal{C}_\Omega^\rho$. We deduce
\[
-\int_{\mathcal{C}_{\Omega}^{\rho}}\frac{\varphi^2}{u} \Delta u dx \geq \mu \int_{\mathcal{C}_{\Omega}^{\rho}} \frac{ \varphi^2}{\abs{x}^2} dx.
\]
By Green's formula, one has
\begin{equation*}
\begin{split}
-\int_{\mathcal{C}_{\Omega}^{\rho}}\frac{\varphi^2}{u} \Delta u dx &= \int_{\mathcal{C}_{\Omega}^{\rho}} \nabla u \nabla\Big(\frac{\varphi^2}{u}\Big) dx \\
&= \int_{\mathcal{C}_{\Omega}^{\rho}}  \left(\frac{2\varphi}{u}\nabla 
\varphi \cdot \nabla u- \frac{\varphi^2}{u^2} |\nabla u|^2\right) dx\\
&\leq \int_{\mathcal{C}_{\Omega}^{\rho}} \abs{\nabla \varphi}^2 dx.
\end{split}    
\end{equation*}
It follows that
$$
\int_{\mathcal{C}_{\Omega}^{\rho}} \abs{\nabla \varphi}^2 dx \geq -\int_{\mathcal{C}_{\Omega}^{\rho}}\frac{\varphi^2}{u} \Delta u dx \geq \mu \int_{\mathcal{C}_{\Omega}^{\rho}} \frac{\varphi^2}{\abs{x}^2} dx.
$$
Furthermore, by Hardy's inequality \eqref{hardy1} and the optimality of the constant $C_{H, \Omega}$  we have 
$\mu \leq C_{H, \Omega}$.
\end{proof}

\noindent In the result below we derive various integral estimates on cone domains involving the kernel $|x|^{-\alpha}$. 

\begin{lemma}\label{esm} Let $\Omega\subset \SS$ be a subdomain, $\rho>0$ and $0\leq \alpha<N$. 

\begin{enumerate}
\item[\rm (i)]
If $f\in L^1_{loc}(\mathcal{C}_\Omega^\rho)$ is positive, then there exists $C>0$ such that
$$ \int_{\mathcal{C}_\Omega^\rho}\frac{f(y)}{|x-y|^{\alpha}} dy \geq C|x|^{-\alpha}\quad\mbox{ for all }x\in \mathcal{C}_\Omega^{2\rho}.$$

\item[\rm (ii)]
If $f:\mathcal{C}_\Omega^\rho\to [0, \infty)$ is measurable and there exists a subdomain $\Omega_0\subset \subset \Omega$ and $c, \beta>0$ such that $f(x)\geq c  |x|^{-\beta}$ in $\mathcal{C}_{\Omega_0}^\rho$, then
$$  
\left\{
\begin{aligned}
& \int_{\mathcal{C}_\Omega^\rho}\frac{f(y)}{|x-y|^{\alpha}} dy =+\infty && \quad\mbox{ if } \; N\geq \alpha+\beta \\[0.05in]
& \int_{\mathcal{C}_\Omega^\rho}\frac{f(y)}{|x-y|^{\alpha}} dy \geq C|x|^{N-\alpha-\beta} && \quad\mbox{ if } \; N< \alpha+\beta
\end{aligned}
\right. \quad\mbox{ in }\; \mathcal{C}_\Omega^{2\rho},
$$
where $C>0$ is a constant. 

\item[\rm (iii)]
If $f:\mathcal{C}_\Omega^\rho \to [0, \infty)$ is measurable and satisfies 
$$
\displaystyle f(x)\leq  c |x|^{-\beta}\log^\tau(1+|x|)\quad\mbox{  in }\; \mathcal{C}_\Omega^\rho,
$$ 
for some $c>0$, $\tau\geq 0$ and $\beta>N-\alpha>0$, then there exists a positive constant $C>0$ such that 
$$ 
\int_{\mathcal{C}_\Omega}\frac{f(y)}{|x-y|^{\alpha}} dy \leq 
\left\{
\begin{aligned}
&C |x|^{N-\alpha-\beta}\log^\tau(1+|x|) &&\quad\mbox{ if }\; \beta<N\\
&C|x|^{-\alpha}\log^{1+\tau} (1+|x|)   &&\quad\mbox{ if }\; \beta=N\\
&C|x|^{-\alpha} \log^\tau(1+|x|) &&\quad\mbox{ if }\; \beta>N
\end{aligned}
\right.
\quad\mbox{ in }\; \mathcal{C}_\Omega^{2\rho}.
$$
\end{enumerate}
\end{lemma}
\begin{proof}
(i) For any $x\in \mathcal{C}_\Omega^{2\rho}$ and $y\in \mathcal{C}_\Omega^{3\rho/2,2\rho}$ we have
$|x-y|\leq |x|+|y|\leq 2|x|$. Thus,  
$$ |x|^{-\alpha}\ast f \geq \int_{\mathcal{C}_\Omega^{3\rho/2,2\rho} }\frac{f(y)}{|x-y|^{\alpha}}dy \geq C\int_{\mathcal{C}_\Omega^{3\rho/2,2\rho} }\frac{f(y)}{(2|x|)^{\alpha}}dy =C |x|^{-\alpha}. $$ 

(ii) We use the spherical coordinates: each $x\in \mathcal{C}_{\Omega_0}$ can be written as 
\begin{equation}\label{sphc}
\begin{cases}
x_1=r\sin\theta_1 \sin\theta_2\cdots \sin\theta_{N-2}\sin\theta_{N-1},\\[0.1in]
\displaystyle x_k=r \cos\theta_{k-1}\prod_{j=k}^{N-1}\sin\theta_{j}, \; 2\leq k\leq N-2,\\[0.1in]
x_{N-1}=r\cos\theta_{N-1},
\end{cases}
\end{equation}
where $r> 0$. Because $\Omega_0$ is a  subdomain of $\SS$,  each $\theta_k$ ($1\leq k\leq N-1$) belongs to a nondegenerate interval $I_k\subset \R$. 

\noindent Let $|y|\geq 2|x|\geq 2\rho$. Then $|x-y| \leq |x|+|y| \leq 3|y|/2$ so that
$$
\begin{aligned}
|x|^{-\alpha}\ast f & \geq c\int_{\mathcal{C}_{\Omega_0}^{2|x|}}\frac{|y|^{-\beta}}{|x-y|^{\alpha}}dy  \geq C\int_{\mathcal{C}_{\Omega_0}^{2 |x|} }  |y|^{-\alpha-\beta} dy\\[0.1in]
&= C\int_{2|x|}^\infty r^{N-1-\alpha-\beta} \, dr=
\begin{cases}
+\infty &\mbox{ if } N\geq \alpha+\beta,\\[0.05in]
C|x|^{N-\alpha-\beta} &\mbox{ if } N< \alpha+\beta,
\end{cases}
\end{aligned}
$$ and the conclusion follows.

(iii) Let $x\in \mathcal{C}_{\Omega}^{2\rho}$. We split the convolution integral as follow:
\begin{equation}\label{esta1}
\int_{\mathcal{C}_{\Omega}^{\rho}} \frac{f(y)}{|x-y|^{\alpha}} dy = \left\{\int_{\mathcal{C}_{\Omega}^{2\abs{x}}} + \int_{\mathcal{C}_{\Omega}^{\abs{x}/2, 2\abs{x}}} + \int_{\mathcal{C}_{\Omega}^{\rho, \abs{x}/2}} \right\} \frac{f(y)}{|x-y|^{\alpha}} dy.
\end{equation}
To estimate the first integral in the right-hand side of \eqref{esta1} let $y\in \mathcal{C}_{\Omega}^{2|x|}$. Then $\abs{x - y} \geq |y|-|x|\geq \frac{\abs{y}}{2}$ so that 

\begin{equation}\label{esta2}
\begin{aligned}
\int_{\mathcal{C}_{\Omega}^{2|x|}} \frac{f(y)}{|x-y|^{\alpha}} dy &\leq C \int_{\mathcal{C}_{\Omega}^{\rho}} \frac{|y|^{-\beta}\log^\tau(1+|y|)}{|x-y|^{\alpha}} dy\\[0.05in]
&\leq C \int_{\mathcal{C}_{\Omega}^{\rho}} \frac{|y|^{-\beta}\log^\tau(1+|y|)}{(|y|/2)^{\alpha}} dy\\[0.05in]
&\leq C\int_{2\abs{x}}^{\infty} t^{N - \alpha - \beta - 1}\log^{ \tau}(1 + t)dt\\[0.05in]
&\leq C 
\abs{x}^{N - \alpha - \beta}\log^{ \tau}(1 + \abs{x}).
\end{aligned}
\end{equation}
For $y\in \mathcal{C}_\Omega^{\abs{x}/2 ,  2\abs{x}}$ we have $\abs{x - y} \leq |x|+|y|\leq 3\abs{x}$. Then
\begin{equation}\label{esta3}
\begin{aligned}
\int_{\mathcal{C}_{\Omega}^{\abs{x}/2, 2\abs{x}}} \frac{f(y)}{|x-y|^{\alpha}} dy &\leq C\abs{x}^{-\beta}\log^{\tau}(1 + \abs{x})\int_{\abs{x - y} \leq 3\abs{x}} \frac{dy}{\abs{x- y}^{\alpha} } \\[0.05in]
&\leq C\abs{x}^{N - \alpha -\beta}\log^{\tau}(1 + \abs{x}).
\end{aligned}
\end{equation}
Finally, for $y\in \mathcal{C}_\Omega^{\rho, \abs{x}/2}$ we have $\frac{\abs{x}}{2} \leq \abs{x - y} \leq \frac{3\abs{x}}{2}$ and
\[\begin{split}
\int_{\mathcal{C}_{\Omega}^{\rho,\abs{x}/2}} \frac{f(y)}{|x-y|^{\alpha}} dy &\leq C\abs{x}^{-\alpha} \int_{\mathcal{C}_{\Omega}^{\rho,\abs{x}/2}} \abs{y}^{-\beta}\log^{\tau}(1 + \abs{y}) dy\\[0.05in]
&= C\abs{x}^{-\alpha}  \int_{\rho}^{\abs{x}/2}t^{N - \beta - 1}\log^{\tau}(1 + t)dt.
\end{split}
\]
It follows that
\begin{equation}\label{esta5}
\int_{\mathcal{C}_{\Omega}^{\rho,\abs{x}/2}} \frac{f(y)}{|x-y|^{\alpha}} dy \leq C\left\{
\begin{aligned}
&\abs{x}^{N - \alpha - \beta}\log^{\tau}(1 + \abs{x}) &&\text{ if } \beta < N,\\
&\abs{x}^{- \alpha}\log^{1 + \tau}(1 + \abs{x}) &&\text{ if } \beta = N,\\
&\abs{x}^{ - \alpha}\log^{\tau}(1 + \abs{x}) && \text{ if } \beta > N.
\end{aligned}\right.
\end{equation}
Now, we combine \eqref{esta1}-\eqref{esta5} to achieve the conclusion.
\end{proof}

\begin{lemma}\label{lg} {\rm (see \cite[Theorem 4.2]{LLM06})}
Suppose $\mu\leq C_{H, \Omega}$ and $u\in C^{2}(\mathcal{C}^\rho_\Omega)$ is such that 
$$
\mathscr{L}_H u\geq 0\quad\mbox{  and } \quad u>0\quad\mbox{  in }\; \mathcal{C}^\rho_\Omega.
$$  
Then, for any subdomain $\Omega_0\subset \subset \Omega$ there exists $c_0> 0$ such that
\begin{equation}\label{gren}
u(x)\geq c_0  |x|^{\gamma_*} 
\qquad\mbox{ in }\; \mathcal{C}^{2\rho}_{\Omega_0}.
\end{equation}
\end{lemma}

\begin{lemma}{\rm (see \cite[Theorem 1.2]{LLM06}) }\label{pap}
Suppose $\mu\leq C_{H, \Omega}$, $q> 1$ and $\beta\in \R$. Then, the inequality
$$
\mathscr{L}_H u\geq |x|^\beta u^q\quad\mbox{ in }\; \mathcal{C}_\Omega^{\rho},
$$
has positive solutions if and only if $\displaystyle q>1+\frac{\beta+2}{|\gamma_*|}$.
\end{lemma}

\section{An a priori estimate}

In this section we establish the following a priori estimate for positive solutions of \eqref{eq0} in $\mathcal{C}^\rho_\Omega$. 

\begin{prop}\label{pconea1}
Let $\rho > 0$ and $u \in C^2(\mathcal{C}_{\Omega}^{\rho})\cap C(\overline{\mathcal{C}_{\Omega}^{\rho}})$ be a positive solution of \eqref{eq0} in $\mathcal{C}_{\Omega}^{\rho}$. For $R >\max\{\rho, e\}$ denote $\eta_R(x) = \eta(\,\abs{x}/R)$ where $\eta \in C_c^1(\R)$ is such that
\[0 \leq \eta \leq 1,\;\;\; \text{supp }\eta \subset [1, 4]\;\;\;\text{ and } \eta \equiv  1 \text{ on } [2, 3].\]
Then, for any  $\lambda > 4$ and $0\leq m <1$, there exists $C = C(N, m, \mu, \alpha) > 0$ such that
\begin{equation}\label{des}
\left(\int_{C_{\Omega}^{\rho}} u^{(p + q - m)/2} \phi (\omega)^{1/2} \eta_R^{\lambda/2} dx\right)^2 \leq C R^{\alpha-2}  \int_{C_{\Omega}^{\rho}} u^{1 - m}\phi (\omega) \eta_R^{\lambda - 2} dx,
\end{equation}
where $\phi > 0$ is the first eigenfunction of $-\Delta_{S^{N - 1}}$ in $\o$ such that $\max_{\oo}\phi = 1$. 
\end{prop}

\begin{proof}
Let $\{\o_{\ve}\}_{0<\ve <1}$ be a sequence of smooth subdomains in $\o$ such that 
$$
\Omega_\ve\subset \overline{\Omega}_{\ve'} \quad \mbox{ for all }\; 0<\ve<\ve'<1
$$
and
$$
\bigcup_{0<\ve<1}\Omega_{\ve}=\Omega.
$$ 
Let $\phi_{\ve}$ denote the first eigenfunction of $-\Delta_{S^{N - 1}}$ in $\o_{\ve}$ which satisfies $\max_{\oo_{\ve}}\phi_{\ve} = 1$. Notice that for all $\omega\in \Omega$ one has $\phi_\ve(\omega)\to \phi(\omega)$ as $\ve\to 0$.  
Define the test function $\varphi$ as
\begin{equation}\label{tf}
\varphi = u^{-m}\phi_{\ve}\eta_R^{\lambda}.    
\end{equation}
Then, supp$\, \varphi=\overline{\mathcal{C}_{{\o}_{\ve}}^{R, 4R}}$ and 
\begin{equation}\label{testt}
\nabla \varphi = -m u^{-m - 1} \phi_{\ve}\eta_R^{\lambda} \nabla u  + u^{-m}\nabla (\phi_{\ve}\eta_R^{\lambda}).
\end{equation}
Multiplying \eqref{eq0} by $\varphi$ and integrating over $\mathcal{C}_\Omega^\rho$ we find
$$
\int_{\mathcal{C}_{{\o}_{\ve}}^{\rho}} (|x|^{-\alpha} * u^p)u^q\varphi\leq 
-\int_{\mathcal{C}_{{\o}_{\ve}}^{\rho}} \varphi \Delta u - \mu \int_{\mathcal{C}_{{\o}_{\ve}}^{\rho}}\frac{u}{\abs{x}^2}\varphi.
$$
By Green's formula and \eqref{testt} it follows that
\begin{equation}\label{qe1}
\begin{aligned}
\int_{\mathcal{C}_{\o_{\ve}}^\rho } (|x|^{-\alpha}  * u^p)u^{q - m}\phi_{\ve}\eta_R^{\lambda} &+ m\int_{\mathcal{C}_{{\o}_{\ve}}^{R, 4R}}u^{-m - 1}\abs{\nabla u}^2 \phi_{\ve}\eta_R^{\lambda} \\
&\leq \int_{\mathcal{C}_{{\o}_{\ve}}^{R, 4R}} u^{-m}\nabla u \nabla (\phi_{\ve}\eta_R^{\lambda}) - 
\mu \int_{\mathcal{C}_{{\o}_{\ve}}^{R, 4R}} \frac{u}{\abs{x}^2}\varphi \\
&= \frac{1}{1 - m}\int_{\mathcal{C}_{{\o}_{\ve}}^{R, 4R}}\nabla u^{1 - m} \nabla (\phi_{\ve}\eta_R^{\lambda}) - \mu 
\int_{\mathcal{C}_{{\o}_{\ve}}^{R, 4R}} \frac{u}{\abs{x}^2}\varphi.
\end{aligned}
\end{equation}
By the divergence formula, we have
\begin{equation}\label{qe2}
\int_{\mathcal{C}_{{\o}_{\ve}}^{R, 4R}}\nabla u^{1 - m} \nabla (\phi_{\ve}\eta_R^{\lambda}) =  \int_{\partial\mathcal{C}_{{\o}_{\ve}}^{R, 4R}} u^{1 - m}\nabla (\phi_{\ve}\eta_R^{\lambda})\cdot \nu d\sigma - \int_{\mathcal{C}_{{\o}_{\ve}}^{R, 4R}}u^{1 - m}\Delta (\phi_{\ve}\eta_R^{\lambda}) dx,
\end{equation}
where $\nu$ is the outer unit normal vector at $\partial\mathcal{C}_{{\o}_{\ve}}^{R, 4R}$. We next claim that 
\begin{equation}\label{qe3}
\nabla (\phi_{\ve}\eta_R^{\lambda}) \cdot \nu \leq 0 \;\;\;\text{ on }\; \partial\mathcal{C}_{{\o}_{\ve}}^{R, 4R}.
\end{equation}
For $x \in \partial\mathcal{C}_{{\o}_{\ve}}^{R, 4R}$ we write $x = (r, \omega)$ where $r = \abs{x}$ and $\omega = \frac{x}{\abs{x}}$. Then, for any $x= (r, \omega) \in \partial\mathcal{C}_{{\o}_{\ve}}^{R, 4R}$ we have 
\[(r, \omega) \in (\{R\} \times \oo_{\ve}) \cup (\{4R\} \times \oo_{\ve}) \cup ([R, 4R], \partial \o_{\ve}).
\]
By the properties of $\eta_R$, we have 
$$
\eta_R (x) = |\nabla \eta_R(x)|=0\quad\mbox{ whenever }\; x= (r, \omega) \in (\{R\} \times \oo_{\ve}) \cup (\{4R\} \times \oo_{\ve}).
$$ 
Thus,
$$
\nabla (\phi_{\ve}\eta_R^{\lambda}) \cdot \nu = 0\;\;\; \text{ on } \;(\{R\} \times \oo_{\ve}) \cup (\{4R\} \times \oo_{\ve}).
$$
Now, let $x = (r, \omega) \in ([R, 4R], \partial \o_{\ve})$ so that $\phi_{\ve}(\omega) = 0$ and then
\begin{equation}\label{samm}
\nabla (\phi_{\ve}\eta_R^{\lambda}) = \phi_{\ve} \nabla \eta_R^{\lambda} + \eta_R^{\lambda}\nabla \phi_{\ve}= \frac{1}{\abs{x}}\eta_R^{\lambda}\nabla_{S^{N - 1}}\phi_{\ve}.
\end{equation}
Recall that $\nabla_{\SS}\phi_\ve(\omega)$ lies in the tangent plane to $\SS$ at $\omega=\frac{x}{|x|}\in \partial \Omega_\ve$. 
Further, $\nu(x)$ and $\nabla_{\SS}\phi_\ve(\omega)$ are both located in the tangent plane to $S^{N-1}$ at $\omega=\frac{x}{|x|}\in \partial \Omega_\ve$  and have opposite direction. This comes from the fact that $\partial \Omega_\ve=\{\phi_\ve=0\}$ is a level set of $\phi_\ve$ and $\nabla \phi_\ve(x)=\frac{1}{|x|}\nabla_{\SS}\phi_\ve(\omega)$  points towards inside of $\Omega_\ve$. Thus,  
$$
\nu(x)=-\frac{\nabla_{\SS}\phi_\ve(\omega)}{ | \nabla_{\SS}\phi_\ve(\omega)|}.
$$ 
Hence, by \eqref{samm} we find
$$
\nabla \big(\phi_\ve \eta^\lambda_R\big)(x) \cdot \nu(x)=-\frac{\eta_R^\lambda(x) }{|x|} | \nabla_{\SS}\phi_\ve (\omega) |\leq 0.
$$
This proves our claim \eqref{qe3}. 
Combining \eqref{qe1}, \eqref{qe2} and \eqref{qe3} we deduce
\begin{equation}\label{qe5}
\int_{\mathcal{C}_{{\o}_{\ve}}^{\rho}} (|x|^{-\alpha}* u^p)u^{q - m}\phi_{\ve}\eta_R^{\lambda} \leq \int_{\mathcal{C}_{{\o}_{\ve}}^{R, 4R}}u^{1 - m}\left\{  \frac{|\Delta (\phi_{\ve}\eta_R^{\lambda})|}{1-m}+|\mu| \frac{ \phi_{\ve}\eta_R^{\lambda} }{|x|^2}  \right\} .
\end{equation}
One has 
\begin{equation}\label{qe6}
\Delta (\phi_{\ve} \eta_R^{\lambda}) = \eta_R^{\lambda}\Delta \phi_{\ve} + 2\nabla \phi_{\ve} \nabla \eta_R^{\lambda} + \phi_{\ve} \Delta \eta_R^{\lambda}.
\end{equation}
By direct computation, we find
\[
\begin{split}
\abs{\eta_R^{\lambda}\Delta \phi_{\ve}} &= \frac{\lambda_1}{\abs{x}^2}\eta_R^{\lambda}\phi_{\ve} \leq CR^{-2}\eta_R^{\lambda} \phi_{\ve} ,\\[0.05in]
\abs{\phi_{\ve} \Delta \eta_R^{\lambda}} &= \lambda \phi_{\ve} \eta_R^{\lambda - 2} \Big| (\lambda - 1)\abs{\nabla \eta_R}^2 + \eta_R \Delta \eta_R\Big| \leq CR^{-2}\eta_R^{\lambda - 2}\phi_{\ve},\\[0.05in]
\nabla \phi_{\ve} \nabla \eta_R^{\lambda} &= \frac{\lambda}{R}\eta_R^{'}\eta_R^{\lambda - 1}\frac{1}{\abs{x}}\nabla_{S^{N - 1}}\phi_{\ve} \cdot \frac{x}{\abs{x}}.
\end{split}
\]
Recall that $\nabla_{S^{N - 1}}\phi_{\ve}$ lies in the tangent plane to the unit sphere $S^{N - 1}$ at $\frac{x}{\abs{x}} \in \o_{\ve}$ and is also orthogonal to $x$. Thus, $\nabla \phi_{\ve} \nabla \eta_R^{\lambda} = 0$ and \eqref{qe6} yields
\begin{equation}\label{qe7}
 |\Delta (\phi_{\ve} \eta_R^{\lambda})| \leq CR^{-2}\eta_R^{\lambda - 2}\phi_{\ve}(\omega)\;\; \text{ in }\;  \mathcal{C}_{\o_{\ve}}^{R, 4R}.
\end{equation}
Now, from \eqref{qe5} and \eqref{qe7} we find 
\begin{equation}\label{qe8}
\int_{\mathcal{C}_{\o_{\ve}}^{\rho}} (|x|^{-\alpha}  * u^p)u^{q - m}\phi_{\ve} \eta_R^{\lambda} \leq CR^{-2}\int_{\mathcal{C}_{\o_{\ve}}^{R, 4R}}u^{1 - m} \phi_{\ve} \eta_R^{\lambda - 2}.
\end{equation}
Furthermore, for $x, y \in \mathcal{C}_{\o_{\ve}}^{\rho, 4R}$, we have $\abs{x - y} \leq \abs{x} + \abs{y} \leq 8R$.
This yields
$$
|x|^{-\alpha}\ast u^p   \geq \int_{\mathcal{C}_{\o_{\ve}}^{\rho, 4R}} \frac{u^p(y)}{|x-y|^{\alpha}} dy  \geq C R^{-\alpha} \int_{\mathcal{C}_{\o_{\ve}}^{\rho, 4R}}u^p(y) dy.
$$
Thus, \eqref{qe8} implies
\begin{equation}\label{tata}
 \left(\int_{\mathcal{C}_{\o_{\ve}}^{\rho, 4R}}u^p\right)\left(\int_{\mathcal{C}_{\o_{\ve}}^{R, 4R}}u^{q - m}\phi_{\ve} \eta_R^{\lambda}\right) \leq C R^{\alpha-2} \int_{\mathcal{C}_{\o_{\ve}}^{R, 4R}}u^{1 - m} \phi_{\ve} \eta_R^{\lambda - 2}.
\end{equation}
Finally, by Holder's inequality we conclude that
$$
\left(\int_{\mathcal{C}_{\o_{\ve}}^{R, 4R}}u^{(p + q - m)/2}\phi_{\ve}^{1/2}(\omega) \eta_R^{\lambda/2} dx  \right)^2 \leq
C R^{\alpha-2} \int_{\mathcal{C}_{\o_{\ve}}^{R, 4R}}u^{1 - m}\phi_{\ve}(\omega)\eta_R^{\lambda - 2} dx.
$$
Letting $\ve\to 0$ in the above estimate we deduce \eqref{des}.
\end{proof}

\section{Proof of Theorem \ref{thm1}}

Part (a) follows directly from Lemma \ref{hrd}. We next assume $\mu\leq C_{H, \Omega}$ and prove part (b) in Theorem \ref{thm1}.

\noindent (i) Assume that \eqref{eq0} has a positive solution $u$ in $\mathcal{C}_\Omega^\rho$ and let $\Omega_0\subset\subset \Omega$ be a proper subdomain of $\Omega$. 
From Lemma \ref{lg} we derive $u\geq c_0 |x|^{\gamma_*}$ in $\mathcal{C}_{\Omega_0}^{2\rho}$ where $c_0>0$ is a constant. 
If  $p\leq \frac{N-\alpha}{|\gamma_*|}$, then,  by Lemma \ref{esm}(ii) it follows that $|x|^{-\alpha} \ast u^p=\infty$ in $\mathcal{C}_\Omega^{2\rho}$ which contradicts condition  \eqref{KKK}. 

\noindent (ii) We divide our argument into two steps. 
\smallskip

\noindent{\bf Step 1: } If $p+q\leq 1$ then \eqref{eq0} has no positive solutions. 

Assume by contradiction that \eqref{eq0} has a positive solution $u$.

If $p+q=1$, then, by taking $m=q\in (0,1)$ in \eqref{tata} we find
$$
\begin{aligned}
\int_{\mathcal{C}^{R, 4R}_{\Omega}}  u^{p} &\geq 
\int_{\mathcal{C}^{R, 4R}_{\Omega}}  u^{p} \phi(\omega)\eta^{\lambda-2}_R\\
&\geq C R^{2-\alpha} 
\Big(\int_{\mathcal{C}^{\rho, 4R}_\Omega} u^p\Big)\Big(\int_{\mathcal{C}^{\rho, 4R}_\Omega } \phi(\omega) \eta_R^\lambda\Big),
\end{aligned}
$$
where $\eta_R$ is the test function defined in the statement of Proposition \ref{pconea1}. The above estimate implies
$$
C\geq R^{2-\alpha} 
\int_{\mathcal{C}^{\rho,4R}_\Omega} \phi(\omega) \eta_R^\lambda\geq CR^{2-\alpha+N} \,,
$$ 
which yields a contradiction as $R\to \infty$. 

Assume now $p+q<1$. By letting $m=p+q\in (0,1)$ in \eqref{des} we find 

\begin{equation}\label{ha1}
\begin{aligned}
\int_{\mathcal{C}^{R,4R}_\Omega} u^{1-m}& \geq \int_{\mathcal{C}^{R,4R}_\Omega} u^{1-m}\phi(\omega) \eta_R^{\lambda-2} \\
&\geq CR^{2-\alpha}  \left(\int_{\mathcal{C}^{R,4R}_\Omega} u^{(p+q-m)/2} \phi(\omega)^{1/2} \eta^{\lambda/2}_R \right)^2\\
&= CR^{2-\alpha}  \left(\int_{\mathcal{C}^{R,4R}_\Omega} \phi(\omega)^{1/2} \eta^{\lambda/2}_R \right)^2\\
& \geq CR^{2N-\alpha+2}.
\end{aligned}
\end{equation}
Let $\Omega_0\subset\subset \Omega$ be a subdomain and $\rho_0>\max\{1,\rho\}$. For $x\in \mathcal{C}_{\Omega_0}$, $|x|=R>\rho_0$, we can find a real number $\delta\in (0,1)$ depending only on $\Omega$ and $\Omega_0$ such that $B_{R\delta}(x)\subset \mathcal{C}_\Omega$. Replacing eventually $\Omega_0$ with a smaller subdomain, we may assume that diam$(\Omega_0)<\delta/6$. Thus, for any $x\in \mathcal{C}_{\Omega_0}$, $|x|=R$ one has (see Figure 2 below):
\begin{equation}\label{ddd}
D:=\mathcal{C}_{\Omega_0}^{R-R\delta/6, R+R\delta/6}\subset B_{R\delta/3}(x)\subset \mathcal{C}_\Omega^\rho.
\end{equation}

\begin{figure}[ht!]
\begin{center}
  \includegraphics[width=3.5in ]{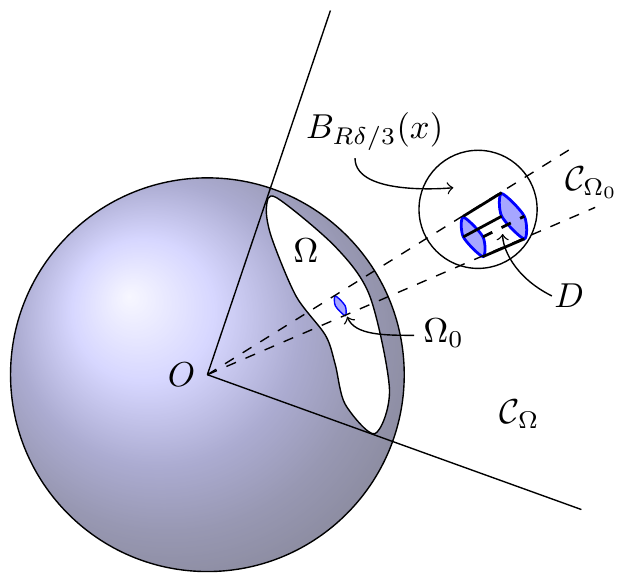}
  \caption{The domain $D$ and the ball $B_{R\delta/3}(x)$ in \eqref{ddd} on which one applies the Harnack inequality.}
  \end{center}
\end{figure}

With a similar argument as in the Proposition \ref{pconea1} in which we replace the pair $(R, 4R)$ by 
$(R-R\delta/6, R+R\delta/6)$ and $\Omega$ by $\Omega_0$, we derive as in \eqref{ha1} that 
\begin{equation}\label{ha2}
\int_{D} u^{1-m} \geq CR^{2N-\alpha+2}.
\end{equation}
By the weak Harnack inequality (see \cite[Theorem 1.2]{Tru67}) and \eqref{ha2}  we find
$$
\begin{aligned}
u(x)&\geq \inf_{B_{R\delta/3}(x)} u \\
&\geq C\left(\frac{1}{(R\delta/3)^N}\int_{B_{R\delta/3}(x)} u^{1-m}\right)^{1/(1-m)} \\
&\geq C\left(\frac{1}{(R\delta/3)^N}\int_{D} u^{1-m}\right)^{1/(1-m)} \\[0.1in]
&\geq CR^{(N-\alpha+2)/(1-m)}.
\end{aligned}
$$
Hence,
$$
u(x)\geq C|x|^\zeta\quad \mbox{ in }\; \mathcal{C}_{\Omega_0}^{\rho_0},
$$
where $\zeta=\frac{N-\alpha+2}{1-m}>0$. This contradicts condition \eqref{KKK} since for $x\in \mathcal{C}_{\Omega}^{\rho_0}$ fixed we have 
$$
\begin{aligned}
|x|^{-\alpha} \ast u^p & \geq 
\int_{y\in \mathcal{C}_\Omega^\rho, \,|y|>|x|} \frac{u^p(y)}{|x-y|^\alpha} dy  
\geq C \int_{\mathcal{C}_{\Omega}^{\rho}, \, |y|>|x|} \frac{|y|^{p \zeta}}{(2|y|)^\alpha} dy =\infty,
\end{aligned}
$$
by using the fact that $\zeta>0$ and $0\leq \alpha<N$. 

\medskip

\noindent{\bf Step 2: } If $p+q>1$ and \eqref{eq0} has a positive solution, then $p+q\geq 1+\frac{N-\alpha+2}{|\gamma_*|}$.

\medskip

Let $m\in (0,1)$ be close to 1 such that $p+q>2-m$. Let 
\begin{equation}\label{este7}
\kappa=\frac{p+q-m}{2-2m}>1\quad\mbox{ and }\quad \kappa'=\frac{\kappa}{\kappa-1}>1.
\end{equation}
With the same notations as in Proposition \ref{pconea1} and using H\"older's inequality we find
\begin{equation}\label{este8}
\int_{\mathcal{C}^{\rho}_\Omega} u^{1-m}\phi(\omega) \eta_R^{\lambda-2} \leq \left(\int_{\mathcal{C}^{\rho}_\Omega} u^{(p+q-m)/2}\phi(\omega)^{1/2} \eta_R^{\lambda/2} \right)^{1/\kappa} \left(\int_{\mathcal{C}^{\rho}_\Omega} \phi(\omega)^{\gamma} \eta_R^{\sigma} \right)^{1/\kappa'},
\end{equation}
where 
$$
\gamma=\frac{2\kappa-1}{2(\kappa-1)}>1\quad\mbox{ and }\quad \sigma=\frac{2(\lambda-2)\kappa-\lambda}{2(\kappa-1)}>0.
$$
Combining \eqref{este8} and \eqref{des} we deduce
\begin{equation}\label{brbr0}
\begin{aligned}
\left(\int_{\mathcal{C}^{\rho}_\Omega} u^{(p+q-m)/2}\phi(\omega)^{1/2} \eta_R^{\lambda/2} \right)^{2-1/\kappa} & \leq CR^{\alpha-2}  \left(\int_{\mathcal{C}^{\rho}_\Omega} \phi(\omega)^{\gamma} \eta_R^{\sigma} \right)^{1/\kappa'}\\[0.05in]
&\leq CR^{\alpha-2+N/\kappa'}.
\end{aligned}
\end{equation}
Take a subdomain $\Omega_0\subset\subset\Omega$. Then, for some constant $c>0$ we have 
\begin{equation}\label{eqwq}
\phi>c>0 \quad \mbox{ in }\; \Omega_0,
\end{equation} and the  estimate \eqref{brbr0} yields
\begin{equation}\label{esteimp}
\left(\int_{\mathcal{C}^{\rho}_{\Omega_0}} u^{(p+q-m)/2} \eta_R^{\lambda/2} \right)^{2-1/\kappa}\leq  CR^{\alpha-2+N/\kappa'}  \quad\mbox{ for } R>2\rho \mbox{ large}.
\end{equation}
We next use the estimate \eqref{gren} in Lemma \ref{lg} together with \eqref{esteimp} to derive 
\begin{equation}\label{qlog}
\begin{aligned}
CR^{\alpha-2+N/\kappa'}   &\geq \left(\int_{\mathcal{C}^{\rho}_{\Omega_0}} u^{(p+q-m)/2} \eta_R^{\lambda/2} \right)^{2-1/\kappa}\\
&\geq \left(c \int_{\mathcal{C}^{\rho}_{\Omega_0}} |x|^{(p+q-m)\gamma_*/2} \eta_R^{\lambda/2} \right)^{2-1/\kappa}\\
&\geq c  \left( R^{(p+q-m)\gamma_*/2} \int_{\mathcal{C}^{R, 4R}_{\Omega_0}} \eta_R^{\lambda/2} \right)^{2-1/\kappa}\\
&\geq  cR^{\big(N+\frac{p+q-m}{2}\gamma_*\big)(2-1/\kappa)} \quad\mbox{ for }R>2\rho \mbox{ large}.
\end{aligned}
\end{equation}
This implies 
$$
\Big(N+\frac{p+q-m}{2}\gamma_* \Big)\Big(2-\frac{1}{\kappa}\Big)\leq \alpha-2+\frac{N}{\kappa'},
$$
where $\kappa$, $\kappa'$ are given in \eqref{este7}. This last estimates finally gives
$$
p+q\geq 1+\frac{N-\alpha+2}{|\gamma_*|}.
$$

\noindent (iii)  Assume $p+q= 1+\frac{N-\alpha+2}{|\gamma_*|}$ and $q>1$. From the estimate \eqref{gren} of Lemma \ref{lg} and Lemma \ref{esm}(ii) we deduce
$$
|x|^{-\alpha}\ast u^p\geq C |x|^{N-\alpha+p\gamma_*}\quad\mbox{ in } \; \mathcal{C}_\Omega^{2\rho}.
$$
Thus, $u$ satisfies
$$
\mathscr{L}_H u\geq |x|^{N-\alpha+p\gamma_*}  u^q\quad\mbox{ in }\; \mathcal{C}_\Omega^{2\rho}.
$$
By Lemma \ref{pap} with $\beta=N-\alpha+p\gamma_*$ the above inequality has no positive solutions if $p+q= 1+\frac{N-\alpha+2}{|\gamma_*|}$.

\qed

\section{Proof of Theorem \ref{thm2}}

We shall discuss the construction of a positive solution to \eqref{eq0} separately for the cases $\mu<C_{H,\Omega}$ and $\mu=C_{H, \Omega}$. 

\medskip

\noindent{\bf Case 1:} $\mu<C_{H,\Omega}$. Then $\gamma_*<-\frac{N-2}{2}<0$ and we can find $\gamma_*<\gamma <-\frac{N-2}{2}$ close to $\gamma_*$ such that $p\abs{\gamma} \neq N$ and
\begin{equation}\label{cond1}
p > \frac{N - \alpha}{\abs{\gamma}}\, ,\;\;\; q> 1 + \frac{2 - \alpha}{\abs{\gamma}}\;\;\; \text{ and }\;\;\;p + q > 1 + \frac{N - \alpha + 2}{\abs{\gamma}}.
\end{equation}
Since $\mu < C_{H, \Omega}$ and $\gamma_*<\gamma <-\frac{N-2}{2}$,  we have  
\begin{equation}\label{inqqq}
\lambda_1 - \mu > \gamma (\gamma + N - 2).
\end{equation}
We construct a positive solution $U$ in the form $U=Cu$, where $C>0$ is a constant and 
\begin{equation}\label{uau}
u(x) = \phi (\omega)\abs{x}^{\gamma}.
\end{equation}
We first check that $u$ given by \eqref{uau} satisfies \eqref{KKK}. Indeed, since $p > \frac{N - \alpha}{\abs{\gamma}}$ and $0\leq \alpha<N$, we have
\begin{equation}\label{macmac}
\begin{aligned}
|x|^{-\alpha}* u^p & = \int_{\mathcal{C}_{\Omega}^{\rho,2|x|}} \frac{u^p(y)}{|x-y|^{\alpha}}  dy+\int_{\mathcal{C}_{\Omega}^{2|x|}} \frac{u^p(y)}{|x-y|^{\alpha}}  dy\\[0.05in]
&\leq \rho^{p\gamma}\! \int_{|x-y|<3|x|} \frac{dy}{|x-y|^{\alpha}}  +2^\alpha \int_{|y|>2|x|} |y|^{-\alpha+p\gamma} dy\\[0.05in]
&\leq C  \int_0^{3|x|}\!\! t^{N-\alpha-1} dt+C \int_{2|x|}^\infty  t^{N-\alpha+p\gamma-1}  dt\\
&<\infty.
\end{aligned}
\end{equation}
Thus, $u$ satisfies the second condition of \eqref{KKK}. 
By \eqref{inqqq} we find 
\begin{equation}\label{estb1}
\mathscr{L}_H u = \phi(\omega)\abs{x}^{\gamma - 2}\left(\lambda_1 - \mu - \gamma(\gamma + N - 2)\right) \geq c\phi(\omega)\abs{x}^{\gamma - 2}\quad \text{ in } \Co^{\rho},
\end{equation}
for some $c > 0$. Applying Lemma \ref{esm}(iii) for $f(x) = u^p(x) \leq \abs{x}^{p\gamma}$ in $\Co^{\rho/2}$, we deduce
\begin{equation}\label{estb2}
\begin{aligned}(\K1 * u^p)u^q \leq C\phi(\omega)^q 
\begin{cases}
\abs{x}^{N - \alpha + (p + q)\gamma} &\text{if } N-\alpha<p\abs{\gamma} < N\\[0.05in]
\abs{x}^{N - \alpha + q\gamma} &\text{if } p\abs{\gamma} > N
\end{cases}
\end{aligned}
\quad\mbox{in }\; \Co^{\rho}.
\end{equation}
Note that by \eqref{cond1} we have 
$$
\gamma-2>N - \alpha + (p + q)\gamma\quad\mbox{and}\quad \gamma-2>- \alpha + q\gamma.
$$
Using the above estimates and $q> 1$, from \eqref{estb1}-\eqref{estb2} we deduce
\begin{equation}\label{estb3}
{\mathscr L}_H u\geq C_1 (|x|^{-\alpha} * u^p)u^q \quad\mbox{in }\;\; \Co^{\rho},
\end{equation}
where $C_1>0$ is a constant. 
Let now $C=C_1^{1/(p+q-1)}$. Then, from \eqref{estb3} we deduce that $U=Cu$ is a positive solution of \eqref{eq0}. 

\medskip

\noindent \textbf{Case 2:} $\mu = C_{H, \Omega}$. Then, the quadratic equation \eqref{eqbeta} has two equal roots given by $\gamma_* = \gamma^* = \frac{2 - N}{2} < 0$. We construct a solution $u$ in the form
\begin{equation}\label{uuu}
u(x) = \phi(\omega)\abs{x}^{\gamma_*} \log^{\tau}(\sigma\abs{x})\quad\mbox{ where }\sigma>\frac{4}{\rho}\quad\mbox{ and }\;\;  0 < \tau <1.
\end{equation}
By direct computation we have
\begin{equation}\label{estb5}
\begin{aligned}
\mathscr{L}_H u &= \tau(1-\tau) \phi(\omega)\abs{x}^{\gamma_*-2}\log^{\tau - 2}(\sigma \abs{x})\;\;\;\text{ in }\mathcal{C}_{\Omega}^\rho.  
\end{aligned}
\end{equation}
Also, by Lemma \ref{esm}(iii) for 
$$
f=u^p\leq C |x|^{p\gamma_*}\log^{p\tau}(1+\abs{x})\quad\mbox{in }\; \mathcal{C}_{\Omega}^{\rho/2},
$$
we find the existence of a constant $C>0$ such that for all $x\in \mathcal{C}_\Omega^{\rho}$ one has 
\begin{equation}\label{estb6}
(|x|^{-\alpha} * u^p)u^q \leq C\phi^{q}(\omega) 
\begin{cases}
\abs{x}^{N - \alpha + (p + q)\gamma_*}\log^{ (p + q)\tau}(1 + \abs{x}) &\text{if }\; N-\alpha<p|\gamma_*|< N,\\
\abs{x}^{-\alpha + q \gamma_*}\log^{1 + (p + q)\tau}(1 + \abs{x}) &\text{if }\; p|\gamma_*|= N,\\
\abs{x}^{-\alpha + q\gamma_*}\log^{q\tau}(1 + \abs{x}) &\text{if }\; p|\gamma_*|> N.
\end{cases}
\end{equation}
Notice that 
$$
\gamma_*-2>\max\{-\alpha+q\gamma_*, N-\alpha+(p+q)\gamma_*\}.
$$
From \eqref{estb5}, \eqref{estb6} and the above inequality we deduce that $u$ satisfies \eqref{estb3}. Similar to the Case 1 above, $U=Cu$ is a solution of \eqref{eq0} for a specific constant $C>0$. 
\qed

\section{Proof of Theorem \ref{thsysinq}}

Part (a) follows directly from  Lemma \ref{hrd}.

(b)(i) Assume \eqref{mainsys} has a positive solution $(u, v)$. Take a proper subdomain $\Omega_0\subset\subset \Omega$.
From Lemma \ref{lg} we derive $v\geq c_0 |x|^{\gamma_*}$ in $\mathcal{C}_{\Omega_0}^{2\rho}$ where $c_0>0$ is a constant. 
If  $p\leq \frac{N-\alpha}{|\gamma_*|}$, then,  by Lemma \ref{esm}(ii) it follows that $|x|^{-\alpha} \ast v^p=\infty$ in $\mathcal{C}_\Omega^{2\rho}$ which contradicts the definition of the solution to \eqref{mainsys}.

\noindent It remains to raise a contradiction if $p, q>0$ satisfy $p+q\geq 2$ and
\begin{equation}\label{sysc00}
p+q<\frac{1}{s}\Big(1+\frac{2}{|\gamma_{*}|}\Big)+\frac{N-\alpha+2}{|\gamma_{*}|}.
\end{equation}

With the same notations as in Proposition \ref{pconea1} let us take $\varphi=\phi_\ve \eta_R^\lambda$ as a test function in \eqref{mainsys}; this corresponds to $m=0$ in \eqref{des}. We use the same approach as in Proposition \ref{pconea1} to find
\begin{equation}\label{csys1}
\left(\int_{\mathcal{C}^{\rho}_\Omega} v^{(p+q)/2} \phi(\omega)^{1/2} \eta^{\lambda/2}_R \right)^2 \leq C R^{\alpha-2} 
\int_{\mathcal{C}^{\rho}_\Omega} u\phi(\omega) \eta_R^{\lambda-2} 
\end{equation}
and
\begin{equation}\label{csys2}
\int_{\mathcal{C}^{\rho}_\Omega} u^{s} \phi(\omega) \eta^{\lambda}_R  \leq C R^{-2} 
\int_{\mathcal{C}^{\rho}_\Omega} v \phi(\omega) \eta_R^{\lambda-2} .
\end{equation}
We next claim that 
\begin{equation}\label{csys3}
\int_{\mathcal{C}^{\rho}_\Omega} v \phi(\omega) \eta_R^{\lambda-2} \leq C R^{N\big(1-2/(p+q)\big)+(\alpha-2)/(p+q)}  \left(\int_{\mathcal{C}^{\rho}_\Omega} u\phi(\omega) \eta_R^{\lambda-2}  \right)^{1/(p+q)}.
\end{equation}
If $p+q=2$, the above estimate follows directly from \eqref{csys1}. Indeed, using the fact that $0<\phi, \eta_R\leq 1$ in $\mathcal{C}_{\Omega}^\rho$ and $\lambda-2>\lambda/2$, from \eqref{csys1} we have
$$
\int_{\mathcal{C}^{\rho}_\Omega} v \phi(\omega) \eta_R^{\lambda-2} \leq 
\int_{\mathcal{C}^{\rho}_\Omega} v^{(p+q)/2} \phi(\omega)^{1/2} \eta^{\lambda/2}_R \leq C R^{(\alpha-2)/2} 
\left(\int_{\mathcal{C}^{\rho}_\Omega} u\phi(\omega) \eta_R^{\lambda-2}\right)^{1/2},
$$
which yields \eqref{csys3} in the case $p+q=2$.

If $p+q>2$, then by H\"older's inequality and \eqref{csys1} we have
$$
\begin{aligned}
\int_{\mathcal{C}^{\rho}_\Omega} v \phi(\omega) \eta_R^{\lambda-2} &\leq \left(\int_{\mathcal{C}^{\rho}_\Omega} v^{(p+q)/2} \phi(\omega)^{1/2} \eta^{\lambda/2}_R  \right)^{2/(p+q)} \left(\int_{\mathcal{C}^{\rho}_\Omega} \phi(\omega)^{\tau_1} \eta^{\tau_2}_R  \right)^{1-2/(p+q)}\\
&\leq C R^{N\big(1-2/(p+q)\big)}\left(\int_{\mathcal{C}^{\rho}_\Omega} v^{(p+q)/2} \phi(\omega)^{1/2} \eta^{\lambda/2}_R  \right)^{2/(p+q)}\\
&\leq C R^{N\big(1-2/(p+q)\big)}\left(R^{\alpha-2} 
\int_{\mathcal{C}^{\rho}_\Omega} u\phi(\omega) \eta_R^{\lambda-2}  \right)^{1/(p+q)}\\
&= C R^{N\big(1-2/(p+q)\big)+(\alpha-2)/(p+q)}  \left(\int_{\mathcal{C}^{\rho}_\Omega} u\phi(\omega) \eta_R^{\lambda-2}  \right)^{1/(p+q)},
\end{aligned}
$$
where
$$
\tau_1=\frac{p+q-1}{p+q-2}>1\quad\mbox{ and }\quad \tau_2=\frac{(p+q)(\lambda-2)-\lambda}{p+q-2}>1.
$$
This concludes the proof of our claim \eqref{csys3}.
By \eqref{csys2} and H\"older's inequality  we have 
\begin{equation}\label{csys4}
\begin{aligned}
\int_{\mathcal{C}^{\rho}_\Omega} u\phi(\omega) \eta_R^{\lambda-2} & \leq \left(\int_{\mathcal{C}^{\rho}_\Omega} u^s \phi(\omega) \eta^{\lambda}_R  \right)^{1/s} \left(\int_{\mathcal{C}^{\rho}_\Omega} \phi(\omega) \eta^{\lambda-2s/(s-1)}_R  \right)^{1-1/s}\\
&\leq CR^{N(1-1/s)} \left(\int_{\mathcal{C}^{\rho}_\Omega} u^s \phi(\omega) \eta^{\lambda}_R  \right)^{1/s} \\
&\leq CR^{N(1-1/s)-2/s} \left(\int_{\mathcal{C}^{\rho}_\Omega} v \phi(\omega) \eta^{\lambda-2}_R  \right)^{1/s}.
\end{aligned}
\end{equation}
Using \eqref{csys3} in \eqref{csys4} we find
\begin{equation}\label{csys5}
\left(\int_{\mathcal{C}^{R, 4R}_\Omega} v \phi(\omega) \eta^{\lambda-2}_R  \right)^{1-1/(s(p+q))}\leq CR^{\tau}\quad\mbox{ for $R>\rho$ large,}
\end{equation}
where
$$
\tau=\frac{1}{p+q}\Big(N\big(1-\frac{1}{s}\big)-\frac{2}{s}+N(p+q-2)+\alpha-2\Big).
$$
Let $\Omega_0\subset\subset \Omega$ be a smooth subdomain. Then, by Lemma \ref{lg}, there exists $c_0>0$ such that 
\begin{equation}\label{c00}
v(x)\geq c_0 |x|^{\gamma_*} \quad\mbox{ in }\; \mathcal{C}_{\Omega_0}^{2\rho}. 
\end{equation}
Using \eqref{csys5} and \eqref{c00}  we find
$$
CR^\tau\geq \left(c \int_{\mathcal{C}^{2R, 3R}_{\Omega_0}} |x|^{\gamma_*}  \right)^{1-1/(s(p+q))}\geq C R^{(N+\gamma_{*})\big(1-1/(s(p+q))\big)} \quad\mbox{ for $R>2\rho$ large.}
$$
This yields
$$
\frac{N+\gamma_{*}}{s}\Big(s(p+q)-1\Big) \leq N\Big(1-\frac{1}{s}\Big)-\frac{2}{s}+N(p+q-2)+\alpha-2,
$$
which implies 
$$
p+q\geq \frac{1}{s}\Big(1+\frac{2}{|\gamma_{*}|}\Big)+\frac{N-\alpha+2}{|\gamma_{*}|}.
$$ 
and this contradicts condition \eqref{sysc00}.

\medskip

(ii) Assume $s > 1+2/|\gamma_{*}|$, $q>1+(2-\alpha)/|\gamma_*|$ and that 
\begin{equation}\label{maas}
p|\gamma_*|>N-\alpha\quad\mbox{ and }\quad p+q>\frac{1}{s}\Big(1+\frac{2}{|\gamma_{*}|}\Big)+\frac{N-\alpha+2}{|\gamma_{*}|}.
\end{equation}

\noindent{\bf Case 1: $\mu<C_{H, \Omega}$.} Note that the second inequality of \eqref{maas} is equivalent to 
$$
\Big(p+q-\frac{1}{s}\Big)|\gamma_{*}|>N-\alpha+2+\frac{2}{s}.
$$
Thus, we may find $\gamma_{*}<b<  -\frac{N-2}{2}$  such that 
\begin{equation}\label{feqp}
s>1+\frac{2}{|b|}\, ,\qquad q>1+\frac{2-\alpha}{|b|} \, ,\qquad p|b|>N-\alpha
\end{equation}
and
\begin{equation}\label{feqpaa}
\Big(p+q-\frac{1}{s}\Big)|b|>N-\alpha+2+\frac{2}{s}.
\end{equation}
By moving $b$ closer to $\gamma_{*}$, we may also assume $p|b|\neq N$. As in the proof of Theorem \ref{thm2}, this  helps us to avoid the logarithmic terms in Lemma \ref{esm}(iii).    From \eqref{feqpaa} we deduce 
\begin{equation}\label{csys7}
\frac{b-2}{s}>N-\alpha+2+(p+q)b.
\end{equation}
On the other hand, from \eqref{feqp}$_1$ we have $s>1+2/|b|$ so
\begin{equation}\label{csys7a}
\frac{b-2}{s}>b.
\end{equation}
Thus, from \eqref{csys7} and \eqref{csys7a} we deduce
$$
\frac{b-2}{s}>\max\Big\{b,  N-\alpha+2+(p+q)b\Big\}.
$$
We thus may find $a\in (\gamma_{*}, b)$ such that 
\begin{equation}\label{csys8}
0>\frac{b-2}{s}>a>\max\Big\{b,  N-\alpha+2+(p+q)b\Big\}.
\end{equation}
Since $\mu<C_{H, \Omega}$ and $\gamma_*<a<b<-\frac{N-2}{2}$ we have 
\begin{equation}\label{csys9}
\lambda_1-\mu >\max\Big\{ a(a+N-2), b(b+N-2)\Big\}.
\end{equation}

Let now $u(x)= \phi(\omega)|x|^a$ and $v(x)= \phi(\omega)|x|^b$. Using \eqref{feqp}$_3$ we have $p|b|>N-\alpha$ and then with the same estimates as in \eqref{macmac} we deduce 
$
|x|^{-\alpha}\ast v^p<\infty$ for all $x\in \mathcal{C}_\Omega^\rho$.  Also, since by \eqref{csys8} one has $(b-2)/s>a$, we obtain 
\begin{equation}\label{baba1}
\begin{aligned}
\mathscr{L}_H v&=\Big(\lambda_1-\mu-b(b+N-2)\Big) \phi(\omega)|x|^{b-2}\\
&\geq  C \phi(\omega)^s |x|^{as} \\
& = C u^s\quad\mbox{ in }\; \mathcal{C}_\Omega^\rho.
\end{aligned}
\end{equation}
It remains to check that $u$ satisfies the first inequality of the system \eqref{mainsys}. To this aim, let us first estimate the nonlocal term $(|x|^{-\alpha}\ast v^p\big)v^q$ in $\mathcal{C}_\Omega^\rho$. We apply Lemma \ref{esm}(iii) for $f=v^p\leq |x|^{pb}$ in $\mathcal{C}^{\rho/2}_\Omega$ and $\beta =p|b|$. We obtain
$$
|x|^{-\alpha}\ast v^p\leq C
\begin{cases}
|x|^{N-\alpha+pb} &\quad\mbox{ if }p|b|<N\\
|x|^{-\alpha} &\quad\mbox{ if }p|b|>N
\end{cases}
\quad\mbox{ in }\; \mathcal{C}_\Omega^\rho,
$$
and then 
\begin{equation}\label{csys10}
\big(|x|^{-\alpha}\ast v^p\big)v^q\leq C
\begin{cases}
\phi(\omega)^q |x|^{N-\alpha+(p+q)b} &\quad\mbox{ if }p|b|<N\\
\phi(\omega)^q |x|^{-\alpha+qb} &\quad\mbox{ if }p|b|>N
\end{cases}
\quad\mbox{ in }\; \mathcal{C}_\Omega^\rho.
\end{equation}
\medskip

\noindent{\bf Case 1a: $p|b|<N$.} Let us observe that by  \eqref{csys8} we have 
$a-2>N-\alpha+(p+q)b$. Using this fact together with \eqref{csys9} and \eqref{csys10} we find 
\begin{equation}\label{baba02}
\begin{aligned}
\mathscr{L}_H u&=\Big(\lambda_1-\mu-a(a+N-2)\Big) \phi(\omega)|x|^{a-2}\\
&\geq C \phi(\omega)^q |x|^{N-\alpha+(p+q)b}\\
&\geq C\big( |x|^{-\alpha}\ast v^p\big) v^q\quad\mbox{ in }\; \mathcal{C}_\Omega^\rho.
\end{aligned}
\end{equation}

\medskip

\noindent{\bf Case 1b: $p|b|>N$.} From  \eqref{csys8} we have 
$0>a>b$ and from \eqref{feqp}$_2$ we know that $q>1+(2-\alpha)/|b|$. Thus, $a-2>b-2>-\alpha+qb$. 
As in Case 1a above, by \eqref{csys9} and \eqref{csys10}  we deduce
\begin{equation}\label{baba2}
\mathscr{L}_H u\geq C \phi(\omega) |x|^{a-2}
\geq C \phi(\omega)^q |x|^{-\alpha+qb} 
\geq C\big( |x|^{-\alpha}\ast v^p\big) v^q\quad\mbox{ in }\; \mathcal{C}_\Omega^\rho.
\end{equation}

From \eqref{baba1} and \eqref{baba02}-\eqref{baba2} there exists a constant $C>0$ such that
\begin{equation}\label{haha}
\begin{cases}
\displaystyle \mathscr{L}_H u \geq C\big( |x|^{-\alpha}\ast v^p\big) v^q\\[0.05in]
\displaystyle \mathscr{L}_H v \geq  C u^s
\end{cases}
\quad\mbox{ in }\; \mathcal{C}_\Omega^\rho.
\end{equation}
Let now $C_1, C_2>0$ be given by
$$
C_1^{s-1/(p+q)}=C^{1+1/(p+q)}\quad \mbox{ and }\quad C_2^{p+q-1/s}=C^{1+1/s}.
$$
Then $(U, V)=(C_1 u, C_2 v)$ is a positive solution of the system \eqref{mainsys}.

\medskip

\noindent{\bf Case 2: $\mu=C_{H, \Omega}$.} Let $u=v$ be given by \eqref{uuu}. We have seen in the proof of Theorem \ref{thm2} that $(u,v)$ satisfies $\mathscr{L}_H u \geq C\big( |x|^{-\alpha}\ast v^p\big) v^q$ in $\mathcal{C}_\Omega^\rho$.

Using $s>1+\frac{2}{|\gamma_*|}$ one has $\gamma_*-2>s\gamma_*$ and a direct calculation yields
$$
\mathscr{L}_H v \geq C \phi(\omega)\abs{x}^{\gamma_*-2}\log^{\tau - 2}(\sigma \abs{x})\geq C\phi(\omega)^s \abs{x}^{s \gamma_*}\log^{s\tau}(\sigma \abs{x})=C u^s \;\;\;\text{ in }\mathcal{C}_{\Omega}^\rho. 
$$
Thus, $(u,v)$ satisfies \eqref{haha} and then as in Case 1 above we deduce that $(U, V)=(C_1 u, C_2 v)$ is a positive solution of \eqref{mainsys}. This concludes our proof.
\qed

\end{document}